\documentclass[12pt,reqno]{amsart}

\usepackage[arrow,matrix,curve]{xy}

\usepackage[dvips]{graphicx} 

\usepackage{amssymb, latexsym, amsmath, amscd, array, hyperref
}

\usepackage{breakurl}   

\newtheorem{theorem}{Theorem}[section]
\newtheorem{lemma}[theorem]{Lemma}

\newtheorem{corollary}[theorem]{Corollary}

\theoremstyle{definition}

\numberwithin{equation}{section}

\newcommand\sys{\mathrm{sys}}

\newcommand\FillRad{\mathrm{FillRad}}

\author{Mikhail G. Katz}\address{M. Katz, Department of Mathematics,
Bar Ilan University, Ramat Gan 5290002 Israel}
\email{katzmik@macs.biu.ac.il}

\begin{document}

\thispagestyle{empty}


\title {A quantitative obstruction to collapsing surfaces}

\maketitle

\begin{abstract}
We provide a quantitative obstruction to collapsing surfaces of genus
at least 2 under a lower curvature bound and an upper diameter bound.

\end{abstract}


\section{Introduction}

S. Alesker posed the following question at MathOverflow \cite{Al16}.
Let~~$(M_i)~$ be a sequence of 2-dimensional orientable closed
surfaces of genus~$g\geq2$ endowed with smooth Riemannian metrics of
Gaussian curvature at least~$-1$ and diameter at most~$D$. By the
Gromov compactness theorem, one can choose a subsequence converging in
the Gromov--Hausdorff (GH) sense to a compact Alexandrov space with
curvature at least~$-1$ and Hausdorff dimension~$0$,~$1$, or~$2$.  Let
us assume that the limit space has dimension~$1$.  Then it is either a
circle or a segment.  Can these possibilities (circle and segment) be
obtained in the limit~$M$ of~$(M_i)$?  We show that these
possibilities cannot occur, and quantify this statement by providing
an explicit lower bound for the filling radius of~$M$.  For related
results see~\cite{Sa17}.

\section{Impossibility of collapse}

We prove the impossibility of collapse in dimension 2, in the
following sense.

\begin{theorem}
\label{t21}
The distance between a strongly isometric map from a closed orientable
surface $M$ of genus $g\geq2$ of Gaussian curvature $K\geq-1$ and
diameter at most $D$ to a metric space $Z$, and a map from $M$ to a
graph in $Z$, is at least~$\frac{\pi(g-1)}{3\sinh D}$.
\end{theorem}

Thus we obtain a quantitative lower bound rather than merely the
nonexistence of Shioya--Yamaguchi-type collapse to spaces of positive
codimension (see \cite{Ya91}, \cite{SY}).

\begin{corollary}
\label{c22}
Let~$D>0$.  GH limits of metrics on a closed orientable surface of
genus~$g\geq2$ with Gaussian curvature at least~$-1$ and diameter at
most~$D$ are necessarily~$2$-dimensional.
\end{corollary}

Recall that the \emph{systole} of a Riemannian manifold~$M$ is the
least length of a noncontractible loop of~$M$.  For an overview of
systolic geometry see \cite{Ka06}.

The \emph{filling radius}~$\FillRad\,M$ of a closed~$n$-dimensional
manifold~$M$ is defined as the infimum of all~$\epsilon>0$ such that
the inclusion of~$M$ in its~$\epsilon$-neighborhood in any strongly
isometric embedding of~$M$ in a Banach space sends the fundamental
homology class~$[M]$ of~$M$ to the zero class, by means of the induced
homomorphism on~$H_n(M)$.  Here the embedding can be taken to be into
the space of bounded functions on~$M$ which sends a point~$p\in M$ to
the distance function from~$p$.  This embedding is strongly isometric
(ambient distance restricted to~$M$ coincides with intrinsic distance
on~$M$) if the function space is equipped with the sup-norm.

\begin{lemma}[Gromov's lemma]
\label{l21}
The systole of an aspherical manifold~$M$ is at most six times the
filling radius of~$M$.
\end{lemma}

\begin{proof}
Consider a strongly isometric embedding of the surface~$M$ into a
Banach space~$B$.  The space~$B$ can be assumed finite-dimensional if
the metric condition is relaxed to a requirement of being bilipschitz
with to a bilipschitz factor arbitrarily close to~$1$; see \cite{KK}.
Suppose~$M$ is ``filled" (in the homological sense) by a chain~$C$ (in
the sense that~$M$ is the boundary of~$C$).  Then the induced
homomorphism~$H_n(M)\to H_n(C)$ sends~$[M]$ to the zero class.
Consider a triangulation of~$C$ into infinitesimal simplices (here the
term ``infinitesimal" is used informally in its meaning ``sufficiently
small" though this could be rendered rigorous as in \cite{15d}).

We argue by contradiction.  Let~$R>0$ be strictly smaller than a sixth
of the systole.  Suppose the chain~$C$ is contained in an
open~$R$-neighborhood of~$M$ in~$B$.  We will retract~$C$ back to~$M$,
while fixing the subset~$M\subseteq C$, contradicting the fact that
the nonvanishing fundamental class~$[M]$ is sent to a zero class
in~$C$.

For each vertex of the triangulation of~$C$, we choose a nearest point
of~$M$.  To extend the retraction to the~$1$-skeleton of~$C$, we map
each edge (of a triangle of the triangulation) to a minimizing path
joining the images of the two vertices in~$M$.  The length of such a
minimizing path is less than~$2R$ (plus the infinitesimal sidelength
of the triangle) by the triangle inequality.  Hence the boundary of
each~$2$-cell of the triangulation is sent to a loop of length at
most~$6R$ (plus an infinitesimal).  Since this length is less than the
systole of~$M$, the map can now be extended to the~$2$-skeleton
of~$C$.

To extend the map to the~$3$ skeleton, note that the universal cover
of~$M$ is contractible and hence~$\pi_2(M)=0$, and similarly for the
higher homotopy groups.  Therefore the skeletal retraction extends to
all of~$C$ inductively.  The contradiction completes the proof of the
lemma.
\end{proof}

\begin{proof}[Proof of Theorem~\ref{t21}]
We exploit Gromov's notion of the filling radius of a
manifold~\cite{Gr83}.  The argument relies only on basic Jacobi field
estimates and basic homotopy theory.  We seek a suitable lower bound
so as to rule out positive-codimension collapse.  Choose a
noncontractible closed geodesic~$\gamma\subseteq M$ of length equal to
the systole~$\sys(M)$.  Consider the normal exponential map
along~$\gamma$.  Using the lower curvature bound, we obtain an upper
bound on the total area of~$M$ as~$2\,\sys(M) \sinh(D)$ where~$D$ is
the diameter.  The bound follows by applying Rauch bounds on Jacobi
fields (this is an ingredient in the proof of Toponogov's theorem);
see e.g., Cheeger--Ebin \cite[Theorem\;5.8, pp.\;97--98]{Ch08}.  The
bound results from comparison with the area of a hyperbolic collar of
width~$D$ around a closed geodesic of the same length as~$\gamma$.
Therefore the systole is bounded below as follows:
\begin{equation}
\label{bound}
\sys(M) \geq \frac{\text{area(M)}}{2\sinh D}.
\end{equation}
Meanwhile the area is bounded below by the Gauss--Bonnet theorem:
\[
\text{area}(M)\geq -\int_M K= 2\pi(2g-2)
\]
where~$g$ is the genus.  Furthermore the filling radius of~$M$ is
bounded below by a sixth of the systole by Gromov's
Lemma~\ref{l21}.   Therefore the bound \eqref{bound} implies
\begin{equation}
\label{e21}
\FillRad(M)\geq\frac16\sys(M)\geq \frac{\text{area}(M)}{12\sinh D}
\geq \frac{\pi(g-1)}{3\sinh D}.
\end{equation}
The theorem now follows from the fact the distance between a strongly
isometric map from $M$ to a metric space $Z$ and a map from $M$ to a
graph in $Z$ is bounded below by the filling radius; see e.g.,
\cite[p.\;127, Example]{Gr83}.  This proves that aspherical surfaces
of curvature bounded below by~$-1$ with diameter bounded above by~$D$
cannot collapse, so that a GH limit is necessarily~$2$-dimensional, as
follows.
\end{proof}

To prove Corollary~\ref{c22}, note that if a metric on $M$ is
sufficiently close to a finite graph $\Gamma$ in the sense of the GH
distance, then the construction of the proof of Lemma~\ref{l21}
produces a map from $M$ to $\Gamma$ which is close to the embedding of
$M$ in $Z$, contradicting the lower bound \eqref{e21}.

\end{document}